\documentclass{amsart}
\usepackage{amssymb}
\usepackage{amsmath,latexsym}
\usepackage[all]{xy}
\newtheorem{thm}{Theorem}[section]
\newtheorem{prp}[thm]{Proposition}
\newtheorem{dfn}[thm]{Definition}
\newtheorem{lmm}[thm]{Lemma}
\newtheorem{crl}[thm]{Corollary}

\def\Oo{{\mathcal O}}

\def\Sc{{\rm S}}

\def\Br{{\rm Br}}

\def\Aut{{\rm Aut}}
\def\End{{\rm End}}
\def\Inj{{\rm Inj}}

\def\Ker{{\rm Ker}}

\def\Hom{{\rm Hom}}

\def\Tr#1#2{{\rm Tr}_{#1}^{#2}}

\def\Hom{{\rm Hom}}

\def\Ind#1#2{{\rm Ind}_{#1}^{#2}}
\def\Res#1#2{{\rm Res}_{#1}^{#2}}

\def\conj{{\rm c}}

\def\Ff{{\mathcal F}}

\def\siB{b}

\def\Oo{{\mathcal O}}

\def\proof{{\it Proof. }}

\begin{document}

\title [Saturated fusion systems and  Brauer indecomposability]
{On   saturated fusion systems and Brauer indecomposability of  Scott modules}
\author[R. Kessar]{Radha Kessar, Naoko Kunugi, Naofumi Mitsuhashi}

\address{Radha Kessar\\
Institute of Mathematics\\
University  of Aberdeen \\
Fraser Noble Building \\
Aberdeen AB24 3UE\\
U.K.}
\email{kessar@maths.abdn.ac.uk}
\address{Naoko Kunugi\\
Department of Mathematics\\
Tokyo University of Science\\
1-3 Kagurazaka\\
Shinjuku-ku\\
Tokyo 162-8601\\
Japan}
\email{kunugi@rs.kagu.tus.ac.jp}
\address{Naofumi Mitsuhashi\\ %%NK%%
Department of Mathematics\\
Tokyo University of Science\\
1-3 Kagurazaka\\
Shinjuku-ku\\
Tokyo 162-8601\\
Japan}
\maketitle
\maketitle

\begin{abstract}
Let $p$ be a prime number, $G$  a finite group,
$P$ a $p$-subgroup of $G$ and  $k$   an algebraically closed  field of characteristic $p$.  
We  study the   relationship between  the  category $\Ff_P(G)$ and  
the behavior  of  
$p$-permutation $kG$-modules  with vertex $P$ under the Brauer construction.  
We give  a  sufficient condition  for   $\Ff_P(G)$ to be a saturated 
fusion system.   We   prove that   for Scott modules with   abelian vertex,  our condition is
also necessary.  In order to obtain our results,  we prove a criterion  for the categories arising from the 
data  of $(b, G)$-Brauer pairs   in the sense of  Alperin-Brou\'e and 
Brou\'e-Puig  to  be saturated fusion systems  on the underlying 
$p$-group.

%$k$  be an algebraically closed field 
%of prime characteristic $p$. Let $A$ be a   
%$p$-permutation $G$-algebra, finite dimensional over $k$  and let 
%$b$ be a primitive idempotent  in the   subalgebra of 
%$G$-fixed points of  $A$.  
%The main result of this paper is a
%sufficient  condition   for the categories arising from the 
%data  of $(b, G)$-Brauer pairs   in the sense of  Alperin-Brou\'e and 
%Brou\'e-Puig  to  be saturated fusion systems  on the underlying 
%$p$-group. We   also study these categories  in the  special case where 
%$A$ is  the $k$-endomorphism  algebra of an indecomposable $p$-permutation 
%$kG$-module. We  use our results to  answer  a  question of N. Kunugi 
%on the indecomposability of Scott modules with respect to the Brauer 
%construction.
\end{abstract}

\section{Introduction}  \label{intro}
Let $p$  be a prime   number and $k$   an algebraically closed field of 
characteristic $p$.  For a finite group $G$,  a $p$-subgroup  $Q$ 
of $G$, and a  finite dimensional  $kG$-module $M$,   the  Brauer quotient
$M(Q)$ of  $M$  with respect to $Q$,  is naturally a  $kN_G(Q)/Q$-module  
and hence  by restriction   is a $kQC_G(Q)/Q$   module 
(see  \cite{broue1}, \cite{brouepuig}, \cite[Section 11]{thevenaz}).  
We  will  say that $M$ is {\it Brauer indecomposable} if  for any $p$-subgroup $Q$  of $G$,  $M(Q)$ 
is   indecomposable  (or zero) as $kQC_G(Q)/Q$-module.

For subgroups $Q, R$ of $G$, let   $\Hom_G(Q, R)$
denote the set of all  group   homomorphisms from $Q$ to $R$ which are 
induced by conjugation by some element of $G$.  For a  $p$-subgroup $P$ of 
$G$, let $\Ff_P(G)$   denote the category whose  objects  are the subgroups of 
$P$; whose   morphism set from   an object $Q$ to an object 
$R$ is the set $\Hom_G(Q, R)$,  and  where  composition of morphisms is the  
usual composition of functions.  We  prove the following result 
(for background on fusion systems and  saturated fusion systems, we refer the  reader to 
the articles \cite{brotolevioliver} and \cite{Linckelmann:2007a}; we note that  we will follow  the notational conventions in \cite{brotolevioliver}  rather than   those  of  \cite{Linckelmann:2007a} in that   all fusion systems will not be assumed to be saturated).

\begin{thm} \label{saturationresultmod}  Let   $G$ be a finite group, 
$P$ a $p$-subgroup of $G$  and  $M$  an indecomposable  
$p$-permutation $kG$-module  with vertex $P$.    If  $M$ is   
Brauer indecomposable, then 
 $\Ff_P(G)$ is a saturated fusion system. 
\end{thm}

The question of  Brauer indecomposability   of  $p$-permutation modules 
(or rather bimodules) plays a role  in   the ``glueing processes'' used  
for   proving   categorical equivalences  between $p$-blocks  of finite groups 
as predicted   by Brou\'e's abelian defect group conjecture 
(see \cite {Kosh/Kun/Waki:2004}, \cite  {Koshitani/Linckelmann:2005}).   
Since  splendid equivalences  between blocks  preserve local structure, 
it is  not unexpected that  there  is a connection between saturation 
and  the  Brauer indecomposability condition. 
Theorem \ref{saturationresultmod}   provides a neat formulation  of the 
connection.

The  converse of  Theorem \ref{saturationresultmod}
does not hold  in general (see remarks after the proof of  Theorem 
\ref{saturationresultmod}).  However,  
in the special case that $M$ is a Scott module, there seems to be 
some control in the reverse direction.  For the definition 
and properties of  Scott modules  we refer the reader to  \cite{broue1}.     
For a finite group $G$  and a $p$-subgroup  $P$ of $G$, we denote by   $\Sc_P(G, k)$   the  $kG$-Scott module  with vertex $P$.

\begin{thm} \label{Scottcyclic} 
Let $P$ be  an abelian $p$-subgroup of  a finite group $G$.   If $\Ff_P(G)$ 
is a saturated fusion system then  $ \Sc_P(G, k)$ is  Brauer indecomposable.   
\end{thm} 

As a corollary, we obtain the  following.

\begin{crl} \label{Scottcycliccor}   Suppose that  the  finite group   $G$ 
has cyclic Sylow  
$p$-subgroups  and let $ P$ be a $p$-subgroup of $G$. Then  
$S_P(G, k)$  is Brauer indecomposable.  
\end{crl}

 Another consequence is  the following result,  of use  for proving 
categorical equivalences  between principal blocks of finite groups.

\begin{crl}\label{bimodulescott}
Let $G_1$ and $G_2$ be finite groups with common abelian Sylow $p$-subgroup  
$P$ and let $\Delta (P) $ be the  diagonal  subgroup 
$\{(x, x)  \, : \, x \in P\} $ of    $G_1 \times G_2 $.
 If $\Ff_P(G_1)=\Ff_P(G_2)$, then $\Sc_{\Delta(P)}(G_1 \times G_2,k)$  
is Brauer indecomposable.
\end{crl}

We do not know whether Theorem \ref{Scottcyclic}   holds without 
the assumption that $P$ is abelian.   Using D. Craven's construction 
in \cite{craven1} of the  Scott modules 
for  the symmetric groups   $S_n$, $ n\leq 6$, we prove the following.  

\begin{prp} \label{symevi} Let $G = S_n$, $n \leq 6 $  and $P$ a $p$-subgroup
of $S_n$.  If  $F_P(G)$ 
is a saturated fusion system, then  $  S_P(G,k)$ is Brauer indecomposable. 
\end{prp}

Let  $A$   be   a   $p$-permutation $G$-algebra, finite dimensional over 
$k$,   and  $b$    a primitive idempotent  in the   subalgebra of 
$G$-fixed points of  $A$.   To each triple $(A, b, G)$, 
there  is associated  a $G$-poset of  Brauer pairs.  These were introduced 
in \cite{alperinbroue} for the case   $A=kG$, 
considered as a $G$-algebra via the  conjugation action of $G$ on itself; 
the general   case was treated in \cite{brouepuig}. 
Roughly  speaking,  an $(A,b,G)$-Brauer pair is a  pair  of the   form  
$ (P, e)$, where $P$ is a $p$-subgroup  of $G$ and 
$e$  is a block of  the Brauer quotient  $A(P)$ of $A$  
in a prescribed relationship with $b$.  For a maximal  object 
$(P, e)$ of  the poset of 
$(A, b, G)$-Brauer pairs, we let  $\Ff_{(P,e)} (A, b,G)$ denote the category 
whose objects  are the subgroups of $P$ and whose morphisms 
are group homomorphisms induced  by the action of $G$ on the underlying poset
(for exact definitions we refer the reader to section 2).
In case  $A=kG $, the results of \cite{alperinbroue} imply that 
$\Ff_{(P,e)} (A, b,G)$  is a saturated fusion system 
(see \cite{Linckelmann:2006b}).  In the general case, it is a consequence 
of \cite{brouepuig}  that  $\Ff_{(P,e)} (A, b,G)$ is a fusion system  in 
the sense of  \cite[Definition 1.1]{brotolevioliver}
(see Proposition \ref{permutationfusion}). However, it is not the case that  
$\Ff_{(P,e)} (A, b,G)$ is  in general  saturated 
(see remarks after the proof of  Theorem \ref{saturationresultmod}  
in section \ref{proofofmodsat}).
Theorem \ref{saturationresultmod}    is a special case   
of   the following  result, due to the first author,  
which gives a sufficiency criterion for saturation.   For an $(A,b,G)$-brauer pair, $(P,e)$, let
$C_G(P, e)$ denote  the    subgroup of  $C_G(P)$  which stabilizes  the  block $e$ of  $A(Q)$ under the 
natural action of $C_G(Q)$  on $A(Q)$.

\begin{thm}\label{saturationresult}   Let  $G$  be a   finite group, $A$  a $p$-permutation 
$G$-algebra,  and  $b$  a primitive idempotent of $A^G$. 
Suppose that 

(i)  $b$ is a central idempotent of  $A$; and 

(ii) For each   $(A, b,G)$-Brauer pair $(Q,f)$  the   idempotent $e$   
is primitive  in  $A(Q)^{C_G(Q, e)} $.
Then for  any maximal $(A, b, G)$-Brauer pair  $(P,e)$, 
$\Ff_{(P, e)}(A, b,G)$ is a  saturated fusion system on $P$.
\end{thm}

We will say that a triple $(A, b, G)$ 
satisfying   conditions (i) and (ii)  of 
Theorem \ref{saturationresult}   is     {\it  a saturated triple} or that  $(A,b,G)$ is  os saturated type.
In this case,  if $G$ and $b$ are clear from the context, we
may also  simply  say that $A$ is of saturated 
type.
If   $A=kG$, then   the primitive idempotents   of $A^{G}$ are  prcisely the  blocks of $kG$,  and it is easy 
to see that $(A, b, G)$ is   a saturated triple (see Remark at the end of  Section 
\ref{saturationsection}), hence
Theorem \ref {saturationresult}  may be viewed as a generalization of 
the  fact that  block fusion systems are saturated.  But the class of 
$p$-permutation $G$-algebras   is  very large.
One  motivation, besides the relevance    to Brauer indecomposability,
for introducing the   notion of  saturated type triples   is   that they 
provide    a  new  source of saturated fusion systems and  
hence may  contribute to our understanding  of  these categories.

The paper  is divided into  four sections. 
In section 2, we  recall the results and definitions of  
\cite{alperinbroue} and \cite{brouepuig}.   Section 3 contains the proof of 
Theorem \ref{saturationresult}. Section 4   deals with  
$p$-permutation modules, and contains the proofs of  
Theorem \ref{saturationresultmod},  Theorem  \ref{Scottcyclic},
Corollary \ref{Scottcycliccor},  Corolllary \ref{bimodulescott} and Proposition \ref{symevi}.

\section {Background and Quoted results} \label{ppermutationsection} 
 
In this section, we set up  notation and recall 
definitions and  background results  on Brauer pairs from the papers 
\cite{alperinbroue} and \cite{brouepuig}.  For notation and terminology regarding fusion systems and saturated fusion  systems, we refer   the reader to 
\cite{Linckelmann:2007a},\cite{brotolevioliver}.

Let $G$ be a finite group, and  let $A$ be a   
$p$-permutation $G$-algebra, finite dimensional over $k$.  Recall that 
$A$ is $p$-permutation if  for any $p$-subgroup $Q$ of $G$ there is a 
$k$-basis  of  $A$ stabilized by $Q$.

\subsection{ } Let  $P$ be a subgroup  of $G$. We denote by 
$A^P$ the subalgebra consisting of the fixed points of $A$ under $P$; 
if $Q$ is a subgroup of  $P$, the map $\Tr{Q}{P}:A^Q \to A^P$ is the 
$k$-linear map defined by the formula 
$\Tr{Q}{P}(a) =  \sum_{x\in P/Q}\,^xa$.  The image of $\Tr{Q}{P} $, denoted  by 
$A_{Q}^P $ is  a two-sided ideal of $A^P$ and  we  denote  by   $ A_{<P}^{P} $ 
the   sum  $  \sum_Q A_{Q}^P $, where $Q$ ranges over the proper subgroups of 
$P$. We denote by $ A(P) $  the  quotient  $A^P/ A_{<P}^{P} $, and 
we  denote by   $\Br_P^A  $ the canonical morphism from $A^P$ onto $A(P)$. 
Recall  from   \cite[Proposition 1.5]{brouepuig} that 
$ A(P)$ is a  $p$-permutation 
$N_G(P)$ algebra.    For $ g\in G$, the map  
which sends an element $\Br_P^A(a)$,  
where $a\in A^P$ to  the element $\,^g(\Br_P^A(a)):=\Br_{\,^gP}^A(\, ^ga) $  
is an algebra isomorphism   from  $A(P)$  to $A(\,^gP)$.

If $ Q \leq   P$ are $p$-groups, then  there exists an algebra 
morphism, $\Br_{P,Q}^A: \Br_Q^A(A^P) \to  A(P)$ such that 
$\Br_{P,Q}^A(\Br_Q^A(a))= \Br_P^A(a) $ for  $ a\in A^P$.
Clearly,  $\, ^g\Br_{P,Q}^A (x) = \Br_{\,^gP,\,^gQ}^A(\, ^gx) $  for any 
$ g\in G$, $x \in \Br_Q^A(A^P) $.  

If, in addition, $Q$ is normal  in $P$, then 
$\Br_Q^A(A^P)= A(Q)^P $ and $\Ker (\Br_{P,Q}^{A})=\Ker(\Br_P^{A(Q)})$. Thus,  
$\Br_{P,Q}^A $ induces an isomorphism  $\siB_{P,Q}^A:  A(Q)(P)\to A(P) $.  
Note that  $\siB_{P,Q}^A$  satisfies and  
is completely determined by the  condition
\begin{equation*} \siB_{P,Q}^A( \Br_P^{A(Q)}(\Br_Q^A(x)))= \Br_{P,Q}^A 
(\Br_Q^A(x)))  )= 
\Br_P^A(x)  \text { \   for all  \ }   x\in A^P. \end{equation*}
Further, $\, ^g\siB_{P,Q}^A(w)= \siB_{\,^gP,Q}^A(\,^gw)  $ for all 
$g \in N_G(Q) $  and $ w \in A(Q)(P)$.

%\begin{dfn}  Let $ A$ be a $p$-permutation $G$-algebra such that  
%$\Br_Q(b) \ne 0 $ and 
%$e$ is a block of $A(Q) $  such that  $\Br_Q(b)e \ne 0 $.   An 
%$(A,G, b)$-Brauer  pair  is a pair 
%$(Q, e)$ where $Q$ is a 
%$p$-subgroup of $G$  such that  $\Br_Q(b) \ne 0 $ and 
%$e$ is a block of $A(Q) $  such that  $\Br_Q(b)e \ne 0 $.
%The group  $G$ acts by conjugation on the set of $(A,G,c)$ Brauer pairs;  
%for a Brauer
%pair $(Q,f)$ we  denote by $N_G(Q,f)$  the stabilizer in $G$ of $(Q,f)$
%under this action.
%\end{dfn}

\subsection { } Let   $b$ be a primitive idempotent of $A^G$.
Recall from \cite[Definition 1.6]{brouepuig} that an $(b,G)$-Brauer pair 
is a pair $(P, e)$ where $P$ is a 
$p$-subgroup of $G$  such that  $\Br_P(b) \ne 0 $ and 
$e$ is a   block  of $A(P) $ such that  
$\Br_P(b)e \ne 0 $.   Here we recall that a  block of a finite 
dimensional algebra is a primitive idempotent of the center of the algebra. 
As we will consider Brauer 
pairs for different algebras simultaneously, we will adopt the more 
cumbersome notation 
$(A, b, G)$-Brauer pair for $(b, G)$-Brauer pair. 

Recall from \cite[Definition 1.6]{brouepuig}  the  notion of inclusion of  
$(A,b,G)$-Brauer pairs:  If  $(Q, f) $ and $(P,e) $ are 
$(A, b, G)$-Brauer pairs, then $(Q, f)\leq (P,e) $  
if $ Q \leq P $ and whenever $i$ is a primitive idempotent of $A^P$ such that 
$\Br_P^A(i)e\ne 0 $, then  $\Br_Q^A(i) f \ne 0 $.  

Let $ (P,e) $   be  an  $(A, b, G) $-Brauer pair and let $ x \in G$. 
The conjugate of $(P,e)$ by $x$ is the $(A,b, G)$-Brauer pair  
$\, ^x(P,e ) := (\, ^xP, \, ^xe )$.   Clearly, conjugation by $x$  preserves inclusion. 

Recall the following fundamental property of inclusion of Brauer pairs
\cite[Theorem 3.4]{alperinbroue},\cite[Theorem 1.8] {brouepuig}. 

\begin{thm} \label{IncTheorem} Let $(P,e)$ be an $(A, b, G)$-Brauer pair, and 
let $Q \leq P$. 

(i) There exists a unique   block $f$    of  $ A(Q)$ such that  $(Q, f) $ is an  $(A,b,G)$-Brauer pair and 
$(Q,f) \leq (P,e) $. 

(ii)  If     $(Q, f)$ is an  $(A,b,G)$-Brauer pair  and $P$ normalizes $Q$,   then   $(Q, f) \leq (P, e) $ if and only if 
$P$ fixes  $f$ and 
$\Br_{P,Q}^A(f)e=  e $.   

(iii) The set of $(A,b,G)$-Brauer pairs is   a  $G$-poset under the action of $G$ defined above.
\end{thm}

Recall also  \cite[Theorem 3.10]
{alperinbroue} and \cite[Theorem 1.14]{brouepuig}).

\begin{thm}\label{relatingBrauerwithtrace}  
Let  $A$ be a  $p$-permutation  $G$-algebra and let 
$b$ be a primitive idempotent  of $A^G$.   Then,   

(i) The   group $G$  acts transitively  on the set of maximal 
$(A,b,G)$-Brauer pairs.

(ii)  Let $(P, e)$  be an $(A, b, G)$-Brauer pair.  The  following are 
equivalent.

(a)   $(P, e)$   is  a maximal $(A,b, G)$-Brauer pair.

(b) $\Br_P^A(b) \ne 0 $ and  $P$ is maximal amongst subgroups $p$-subgroups
 $Q$ of $G$ 
 with the property that  $\Br_Q^A(b) \ne 0 $.

(c)  $ b \in \Tr{P}{G} (A^P)$   and  $P$ is minimal
amongst subgroups $H$ of $G$ such that $ b \in  \Tr{H}{G} (A^H)$.
\end{thm}

The  equivalence   of ii(b) above with  ii(a)   
is not  explictly stated in \cite[Theorem 1.14]{brouepuig}, 
but     is  an  immediate consequence of (i). 
For clearly, if $P $ satisfies   ii(b), then $(P,e) $-is a maximal 
$(A, b, G)$-Brauer pair. Conversely, if $(P,e)$ is a maximal 
$(A, b, G)$-Brauer pair and  $ P \leq R $ is such that $\Br_R^A(b) \ne 0 $, 
then there exists some  block $t$ of $A(Q)$ such that $(R, t)$ is an 
$(A, b, G)$-Brauer pair. Let $(S,u) $ be a maximal $(A, b, G)$-Brauer pair with
$(R,t) \leq (S, u) $. Then  by (i), $(P,e) $ and $(S,u) $ 
are  $G$-conjugate. In particular, $ |P|= |S| \geq |R| \geq |P| $, hence $P= R$.

\bigskip

If $Q, R $  are subgroups of  $G$ and  $g\in G $ is such that  
$\,^gQ \leq R  $,  then   $\conj_g: Q\to R  $ denotes  
the map  which sends an element $x $ of $Q$ to the element   
$\,^gx:=gxg^{-1} $ of $R$.

\begin {dfn}   \label{ppermutationfusiondefn}   
Let $(P, e_P)$ be a maximal $(A, b, G)$-Brauer pair.   For each subgroup 
$Q$ of $P$, let $ (Q,e_Q)$ be  the unique $(A, b, G)$-Brauer pair 
such that $(Q, e_Q) \leq  (P, e_P)$.  The category  
$\Ff_{(P,e_P)} (A, b,G)$  is  the category   whose  objects are the  subgroups 
of $P$, whose morphisms are   given  by
$$\Hom_{\Ff_{(P,e_P)}(A,b,G)}(Q,R):=\{\conj_g:Q\to R |g\in G, \,^g(Q,e_Q)
\le(R,e_R)\}  $$  for $Q, R \leq   P$,
and where composition of morphisms  is the usual composition  of functions.
\end{dfn}

For any $ Q \leq R$, the inclusion  map from
$ Q$   to $R $  is a morphism  in 
$\Ff_{(P,e_P)} (A, b,G)$. In particular, the identity map  $ Q \to Q$  
is a morphism in 
$\Ff_{(P,e_P)} (A, b,G)$    and if $ R , S\leq P$ and $g, h \in G$ are such 
that  $\,^g(Q,e_Q)
\le(R,e_R) $ and $\,^h(R,e_R)
\le(S,e_S)$, then 
$$ \,^{hg}(Q,e_Q)  \leq  \,^h(R,e_R) \le   (S,e_S), $$ 
so $\Ff_{(P,e_P)} (A, b,G)$ is a category.    By the   uniqueness of 
inclusion of Brauer pairs
for $ Q, R \leq P $  and 
$g \in G$,    $\,^g(Q,e_Q)
\le(R,e_R) $ if and only if $\, ^g Q \le R $ and  $  \,^ge_Q = e_{\,^g Q}$  
  and this in turn  holds if and only if $\, ^g Q \le R $ and  $\,^g(Q,e_Q)
\le(P,e_P) $.    
Thus  if  $x \in P$,  then since $e_P$ is fixed  by $P$,   
$\,^xe_P = e_P$. Hence,  for  $Q \leq P$,
$$  \, ^x(Q, e_Q)  \leq \, ^x(P, e_P) =  (P,e_P). $$  So, 
 whenever $\,^x Q \leq R $,  then
$ \conj_x : Q\to R  $ is a morphism in   $\Ff_{(P,e_P)} (A, b,G)$.

Also, note  that if $Q, R \leq P$ and $ g \in  G $ are such that 
$\,^g(Q,e_Q)\le (R,e_R) $, then  $ \conj_g : Q \to R$  factors as 
$\conj_g :  Q \to \, ^gQ $  
followed by the inclusion of $\, ^gQ $ into $R$.    Summarizing the  above 
discussion gives the following proposition,  the  last  statement  of which    is 
immediate  from  the fact that  any two  maximal 
$(A,b, G)$-Brauer pairs  are  $G$-conjugate.

\begin{prp} \label{permutationfusion} Let $A$  be a $p$-permutation 
$G$-algebra,  $b$  a primitive idempotent of 
$A^G$  and $(P, e_P)$    a maximal 
$(A,b,G)$-Brauer pair. 
Then  $\Ff:=\Ff_{(P,e_P)} (A, b,G)$ satisfies the following.

(i) $\Hom_P(Q,R) \subseteq  \Hom_{\Ff}(Q,R) \subseteq \Inj (Q,R) $ for all  
$Q,R \le P$.

(ii)  For any  $\phi \in \Hom_{\Ff}(Q,R) $, the induced isomorphism  
$Q \cong \phi(Q)$  and its inverse are  morphisms in $\Ff $ and its inverse 
are morphisms in $\Ff$. In particular,  every morphism in $\Ff$ factors 
as an isomorphism in $\Ff$  followed by an inclusion  in $\Ff$.

Thus,  $\Ff $ is a fusion system in the sense of 
\cite[Definiton 1.1]{brotolevioliver}.  If $ (P',e_{P'})$ 
is another  maximal $(A, b,G)$-Brauer pair, then 
$\Ff_{(P', e_{P'})} (A,b,G)$ is isomorphic to $\Ff_{(P, e_P)}(A, b, G)  $.
\end{prp}

\section{Proof of Theorem \ref{saturationresult}} 
\label{saturationsection}
Throughout this section,   $G$ will denote a finite group, $A$   a $p$-permutation $G$-algebra, and 
$b$  a primitive idempotent of $A^G$. 
Recall from the introduction that 
$(A, b, G)$  is a saturated triple  if conditions  (i) and (ii) of Theorem  \ref{saturationresult}  hold.  Thus, 
we will prove that if $(A, b, G)$ is   a saturated triple, then 
$\Ff_{(P,e_P)}(A, b, G)$ is  saturated for any maximal $(A, b, G)$-Brauer pair 
$(P, e_P)$. We need some  preliminary results.

\begin{lmm}\label{fullynormalisedSylowgen} Let  $H$ be a finite group and let $B$ be an $H$-algebra. 
 Let $R$ be  a subgroup  of $H$ and let $C$  be a normal subgroup of $H$.  Suppose  that
 $1_B \in \Tr{R}{H}(B^R)$ and  $1_B$ is  primitive in $B^C$. Then, $RC/C$  contains a Sylow $p$-subgroup of $H/C$. 
 \end{lmm}

\proof  Let $b \in B^R$ be such that 
$$1_B  = \Tr{R}{H}(b) = \Tr{RC}{H} (\Tr{R}{RC}(b)),$$
and set  $u:=\Tr{R}{RC}(b)$. Then,
$u \in  B^{RC}\subseteq B^C$. By hypothesis, the identity $1_B= 1_{B^C}$ of  $B^C$ is the only idempotent of $B^C$. In other words, $B^C$ is a local algebra which means that $J(B^C)$ has co-dimension $1$ in  $B^C$.  Thus, we may write
$u = \lambda 1_B + v $ for some $\lambda \in k$ and $v \in J(B^C) $.
Thus, $$1_B  = \Tr{RC}{H} ( \lambda 1_B + v  )  =   [H: RC]\lambda 1_B + \Tr{RC}{H}(v).$$
Now, since $C$ is normal in $H$, $H$ acts  on $B^C$  and hence on $J(B^C)$.  In particular,
$\Tr{RC}{H}(v) \in J(B^C) $.    But $1_B \notin J(B^C)$.  Hence, it follows  from the above  displayed equation  that  $[H:RC]$  is not divisible by $p$, proving the lemma.

For $(A,b,G)$-Brauer pairs  $(Q,f) \leq (P,e) $, set  
\iffalse
$$ T _G((Q,f), (P,e)):= T_{(A, b, G)}((Q,f), (P,f))    
:=\{ x \in G \, : \,  \text{ \ such that   \  }   \,^x(Q,f) \le (P,e) \}, $$
\fi
$$ N_ G(P, e) :=  N_{(A, b, G)}((P,e))   :=\{ x \in G \, : \,  
\text{ \ such that \  }   \,^x(P,e) =(P,e) \}, $$
and $$ C_G(P,e) := N_G(P,e)\cap C_G(Q). $$  

\iffalse
For an $(A,b,G)$-Brauer pair, $(Q,e)$,  
let  $C_G(Q,e)$  denote the   subgroup of $C_G(Q)$  stabilising the pair 
$(Q,e)$ and  let  $N_G(Q,e)$  denote the  subgroup of 
$N_G(Q)$  stabilising the pair $ (Q,e)$. 
\fi

\iffalse
The   cases  of  special  
relevance in block theory
 are $A=kG$,  and  $A=kN$   for 
$N$  a normal subgroup of $G$ and these   
have  been  discussed  before(cf. \cite{linckelsol},\cite{kessarstancu} 
\cite{kessarlausanne}).  
\fi

\iffalse

By definition, if $(Q, e)$ ia an $(A,b,G)$-Brauer pair, then 
$e$ is a block of  $A(Q)$.  Now $A(Q)$ is a  $N_G(Q)$-algebra, 
hence   an $C_G(Q,e)$-algebra and by definition of $ C_G(Q,e_Q)$, 
$ e \in A(Q)^{C_G(Q,e)}  $.  

\begin{dfn}\label{saturatedtypedefn}   The triple
$(A,b,G)$  is  \emph{ a saturated triple}    if  the 
following two conditions hold.

(i)   $b$ is a central idempotent of  $A$.

(ii)   For each   $(A, b,G)$-Brauer pair $(Q,e)$  the   idempotent 
$e$   is primitive  in  $A(Q)^{C_G(Q, e)} $.

\end{dfn}

The reason for the above terminology is the following result.

\begin{thm}\label{saturationresult}  Suppose that $(A,b, G)$ is  a 
saturated triple. Then, $\Ff_{(P, e_P)}(A, b,G)$ is a  saturated 
fusion system on $P$ for any maximal  $(A, b, G)$-Brauer pair $(P, e_P)$.    
\end{thm}

The proof requires some lemmas.
\fi
\begin{lmm}  \label{centralbrauercondition}  Let  $H$  be a finite  group,  
$B$  a  $p$-permutation $H$-algebra and  $e$ a   primitive  idempotent  of   
$B^H$.    If $e  \in Z(B)$,  then   
for a  $p$-subgroup
$Q$ of $H$ and   a block   $f$ of $B(Q)$, $(Q,f)$ is an 
$(B, e, H)$-Brauer pair if and only if   $\Br_Q^B(e)f =f $. 
\end{lmm} 
  
\begin{proof}    Suppose  that  $e  \in Z(B)$ and let 
$Q$ be a $p$-subgroup of $H$.   Since 
$$  Z(B)  \cap B^H \subseteq   Z(B) \cap  B^Q \subseteq Z(B^Q), $$  
$e$ is a central idempotet of $B^Q$.  Hence,    
either  $\Br_Q^B(e) =0 $ or 
$\Br_Q^B(e) $ is a central idempotent
of $B(Q)$  and   for any block $f$  of  $B(R)$,  either 
 $\Br_Q^B(e) f =f $,    or $\Br_Q^B(e)f =0 $. The result follows.  
\end{proof}

 For the next result, we note the following. For an $(A,b,G)$-Brauer pair $ (Q,e)$, 
$A(Q)$ is a $N_G(Q)$-algebra and $e$  is an idempotent of 
$A(Q)^{N_G(Q,e)} $. Thus, if $e$ is primitive in  
$ A(Q)^{C_G(Q,e)}$, then $e$ is  a primitive idempotent of 
$A(Q)^H$ for any $H$ such that $C_G(Q,e_Q) \leq H \leq N_G (Q, e_Q)$ and
it makes sense to speak of $(A(Q),e, H)$-Brauer pairs.

\begin{lmm}\label{Qtrivial}    Suppose  that 
$(Q,e)$   is  an  $(A,b,G)$-Brauer pair  such that $e$ is primitive 
in $A(Q)^{C_G(Q,e)}$  and let  $H$   be a subgroup of $G$ 
with  $C_G(Q, e) \leq  H \leq  N_G(Q,e )$.   

(i)  The $H$-poset of $(A(Q), e,  H)$-Brauer pairs is the  
$H$-subposet of 
$(A(Q), e,   N_G(Q,e))$-Brauer pairs  consisting of 
those pairs whose first component is 
contained in $H$. 

(ii)  The  map  $$(R,\alpha )  \to  (QR,  \alpha)$$  
is an  $H$-poset  homomorphism  from  
the  set of  $(A(Q), e, H) $-Brauer pairs to the set of 
$(AQ), e, QH)$-Brauer pairs and induces a  bijection  between 
the  set of  
$(A(Q), e, H) $-Brauer pairs whose first component contains  $Q \cap H$ and 
the  set of
$(A(Q), e, QH)$-Brauer  pairs  whose first component contains $Q$.

(iii)   If $Q \leq  H$, then  $(Q, e) $ is  the unique   
$(A(Q), e, H)$-Brauer pair    with first component $Q$ and 
 $(Q, e)  $ is contained   every maximal $(A(Q), e, H)$-Brauer pair.
\end{lmm}

\begin{proof}  (i)
This is  immediate from the definitions.

(ii)  Since $Q$ acts trivially on $A(Q)$,  
for any   $p$-subgroup $R$ of $H$,  $A(Q)^R= A(Q)^{QR}$ and $\Br_{R}^{A(Q)} = 
\Br_{QR}^{A(Q)}$.  The first assertion is  immediate from this observation.  
The second assertion follows from the  first and the fact that 
$ R \to  QR$   is a bijection between subgroups of  $H$ containing 
$Q \cap H$ and subgroups of   $QH$ containing  $Q$.

(iii)  By hypothesis,  $A(Q)^Q= A(Q) $. Hence,  
$A(Q)_{<Q}^Q = 0 $ and  $\Br_Q^{A(Q)}$ is  the identity map on $A(Q)$. 
Thus,  the set of 
$(A(Q),e, H)$-Brauer pairs with first component $Q$  consists precisely of  
the pairs $ (Q, \alpha)$, where $\alpha$ is a  block   of $A(Q)$ such that 
$e\alpha \ne 0 $. Since   
$e$ itself is    a  block of $A(Q)$ and any two distinct blocks of 
$A(Q)$ are orthogonal, it follows that  $(Q, e)$  is an $(A(Q), e, H)$-Brauer 
pair and  that it is the  unique one with first component $Q$.    Since 
$\,^h(Q,e) =(Q,e)$  for all $h\in  H$  and  by Theorem 
\ref{relatingBrauerwithtrace}(a) $H$ acts transitively  on the set of maximal 
$(A(Q), e, H)$-Brauer pairs, $(Q,e)$ is contained in every maximal 
$(A(Q), e, H)$-Brauer pair.
\end{proof}

To prove that  a fusion system of a finite group $G$ on a Sylow $p$-subgroup  
$S$  of the group is saturated one  applies Sylow's theorem   
to   the local subgroups   $N_G(Q)$   and $N_S(Q)C_G(Q)$  of $G$,  for  $Q$  a $p$-subgroup of $G$.
The proof of Theorem \ref{saturationresult}  is based on the  same idea with    
triples of the form $(A(Q),e, N_G(Q,e_Q))$,  $(A(Q),e, N_P(Q)C_G(Q,e_Q))$ 
playing the role   of local subgroups and 
Theorem \ref{relatingBrauerwithtrace} and Lemma \ref{fullynormalisedSylowgen}  
playing the role of Sylow's theorem.    The next    result allows us to pass  back and forth 
between $(A, b, G)$-Brauer pairs and  $(A(Q),e, H)$-Brauer pairs.
Recall the  isomorphisms  $\siB_{R, Q}^A :A(Q)(R) \to  A(R)$   for 
$p$-subgroups $Q \unlhd R$ of $G$  introduced  at the  end of subsection 2.1.

\begin{lmm}  \label{localcompatibility}   Suppose that 
$(Q,e)$   is  an  $(A,b,G)$-Brauer pair  such that $e$ is primitive 
in $(A(Q))^{C_G(Q,e)}$  and let  $H$   be a subgroup of $G$ 
with  $QC_G(Q, e) \leq  H \leq  N_G(Q,e )$.

The  map  $$(R,\alpha )  \to  ( R,  \siB_{R,Q}^A(\alpha))$$  
is an $H$-poset isomorphism between the  subset of  
$(A(Q), e, H) $-Brauer pairs  consisting 
of those pairs whose first component contains  $Q$, and the subset of  
$(A, b, G)$-Brauer  pairs   containing  
$(Q,e)$  and  whose  first component is contained in $H$.  

In particular, $H$ acts transitively  on the subset of $(A, b, G)$-Brauer 
pairs which are maximal with respect to containing $(Q,e)$ and  having 
first component  contained in $H$.
\end{lmm}

\begin{proof} Let 
${\mathcal P}_1$  be the  subset of  $(A(Q), e, H) $-Brauer pairs  consisting 
of those pairs whose first component contains  $Q$, and let 
${\mathcal P}_2$  be the subset of $(A, b, G)$-Brauer  pairs   containing  
$(Q,e)$  and  whose  first component is contained in $H$.  Since  
$H \leq N_G(Q,e) \leq N_G(Q)$,  ${\mathcal P}_1 $ and ${\mathcal P}_2 $ 
are  $H$-posets.
Now let $ Q \le R \le H$,   and let $\alpha $   be a   block  of  $A(Q)(R)$. 
By Lemma \ref {centralbrauercondition},  $ e=\Br^A_Q(b)e $, hence
 $$ \Br_{R,Q}^A(e)=   \siB_{R,Q}^A (\Br_{R}^{A(Q)}(e))  =   
\siB_{R,Q}^A  (\Br_{R}^{A(Q)} (\Br^A_Q(b)e))  = \Br_R^A(b) \Br_{R,Q}^A(e).$$   
Suppose first  that  $(R,\alpha)$ is an 
$(A(Q),e, H)$-Brauer pair. By Lemma \ref {centralbrauercondition},
$\alpha= \Br_{R}^{A(Q)}(e)\alpha $.  Applying  $\siB_{R,Q}^A  $ 
to both sides of this equation, and using 
the displayed equation  above, we get that 
$$ \siB_{R,Q}^A(\alpha) = \Br_{R,Q}^A(e) \siB_{R,Q}^A(\alpha)   
= \Br_R^A(b) \Br_{R,Q}^A(e) \siB_{R,Q}^A(\alpha).    $$  
In particular,  $\Br_R^A(b) \siB_{R,Q}^A(\alpha) \ne 0 $, whence  
$(R, \siB_{R,Q}^A(\alpha) )$ is  an $(A,b,G)$-Brauer pair.
By   Theorem \ref{IncTheorem} and the first equality above, 
$(Q, e) \leq (R,  \siB_{R,Q}^A(\alpha))$ as   $(A,b,G)$-Brauer pairs. 

Conversely, if $(Q, e) \leq (R,  \siB_{R,Q}^A(\alpha))$, then   again 
by   Theorem \ref{IncTheorem},  
$\siB_{R,Q}^A(\alpha) = \Br_{R,Q}^A(e) \siB_{R,Q}^A(\alpha)  $. Applying the 
inverse  of $\siB_{R,Q}^A $ yields that $\alpha= \Br_{R}^{A(Q)}(e)\alpha $,  
hence that  $(R,\alpha)$ is an 
$(A(Q),e, H)$-Brauer pair.
This shows that  $(R,\alpha) \rightarrow  (R,\siB_{R,Q}^A(\alpha))  $ 
is a bijection between  ${\mathcal P}_1 $ and ${\mathcal P}_2 $.

We show that the bijection is inclusion preserving.  
Let $(R,\alpha )$ and $(S,\beta)$ be  $(A(Q),e, H)$-Brauer pairs   with
$ Q \lhd  R \leq S $.  By    Theorem  \ref{IncTheorem}, 
it suffices to consider the case  that 
$ R \unlhd S$.   Clearly,  $\alpha $ is $S$-stable if and only if 
$\siB_{R,Q}^{A}(\alpha) $ is $S$-stable.   Further, the restrictions 
of the  maps
$ \siB_{S,Q}^A\circ \Br_{S,R}^{A(Q)} \circ \Br_R^{A(Q)}\circ\Br_Q^R  $ and 
$\Br_{S,R}^A \circ \siB_{R, Q}^A\circ \Br_R^{A(Q)}\circ\Br_Q^A $ to 
$A^S$ both equal  $\Br_S^A$.  Since  $\Br_R^{A(Q)}\circ\Br_Q^R (A^S) = A(Q)(R)^S $, it follows that 
$ \siB_{S,Q}^A\circ \Br_{S,R}^{A(Q)}  $ is equal to the restriction of 
$\Br_{S,R}^A \circ \siB_{R, Q}^A$ to  $A(Q)(R) ^S $. In particular, 
$\Br_{S,R}^{A(Q)} (\alpha)\beta = \beta $ if and only if 
$\Br_{S,R}^{A} ( \siB_{R,Q}^{A}(\alpha) )\siB_{S,Q}^{A}(\beta) =
\siB_{S,Q}^{A}(\beta)$. Thus, by Theorem \ref{IncTheorem}
$(R,\alpha)\leq  (S, \beta)$  if and only if  
$(R,\siB_{R,Q}^{A}(\alpha))\leq  (S, \siB_{S,Q}^{A}(\beta))$, 
and the bijection is inclusion preserving.  Since $Q $ is normal in $H$,  
$$ \siB_{\,^hR, Q} ^A (\, ^h \alpha)  =   \siB_{\,^hR, \,^hQ} ^A (\, ^h \alpha)
=   \,^h\siB_{R, Q} ^A ( \alpha)  $$ for all $h \in H$,  all  
$p$-subgroups  $R$ of $G$ containing  $Q$ as normal subgroup and  all 
$\alpha \in   A(Q)(R)$, and hence the above bijection  is compatible with the 
$H$-action on ${\mathcal P}_1 $ and ${\mathcal P}_2 $. This proves that the 
given map is an isomorphism of $H$-posets. In particular,  the  map induces a  
bijection between the    set of maximal elements    of 
${\mathcal P_1}$ and ${\mathcal P_2}$.
But  by Lemma \ref{Qtrivial} (c),  the set of maximal elements in  ${\mathcal P}_1 $ 
is    precisely the  set of maximal $(A(Q), e, H)$-Brauer pairs. 
The  final assertion follows from this  and  from the fact that $H$ acts 
transitively on the 
set of maximal $(A(Q), e, H)$-pairs (see \ref{relatingBrauerwithtrace} (a)).
\end{proof} 
 
We will  prove  Theorem \ref{saturationresult}  by  using the the saturation   axioms given by Robertson and Schpectorov     in \cite{robertson-shpectorov} .   For this we recall the following    terminology:  If $\Ff$  is a   fusion system on  a finite $p$-group $P$, then a subgroup  $Q$ of $P$ is  {\it fully  automized}  if $\Aut_{P}(Q)$   is  a Sylow $p$-subgroup of $\Aut_{\Ff}(Q)$   and $Q$ is {\it receptive}   if for  any   isomorphism $\varphi  \,: \, R \to Q $   in $\Ff$,   there exists a 
morphism $\hat \varphi  \, : \,  N_{\varphi}  \to  P $  in $\Ff$  such that   ${\rm  Res} |_R \hat \varphi  =\varphi  $, where $N_{\varphi} $   is  the subgroup of $N_P(R)$ consisting of   those  elements  $z \in N_P(R)$  such that  
$\varphi \circ   c_z  =  c_x \circ \varphi  $   for some $x \in N_P(Q)$.

\begin{lmm}  \label{fulautrec}   
Suppose that   $(A,b,G)$ is a saturated triple and let $(P, e_P)$ be 
a maximal $(A,b, G)$-Brauer pair.  For each $Q\leq P$ 
let $e_Q$ be the unique block of $A(Q)$ such that    
$(Q, e_Q) \leq (P, e_P)$ and  let $\Ff= \Ff_{(P, e_P)}(A,b, G) $.
If $Q\leq P$ is such that  $(N_P(Q), e_{N_P(Q)}) $ is 
maximal amongst  $(A, b, G)$-Brauer pairs   $(R, f)$ with    
$(Q, e_Q) \leq   (R, f)$ and $ R \leq N_G(Q, e_Q) $, then   
$Q$ is fully $\Ff$-automised and $\Ff$-receptive.\end{lmm}

\begin{proof}  Suppose that  $(N_P(Q), e_{N_P(Q)}) $ is 
 maximal amongst  $(A, b, G)$-Brauer pairs   $(R, f)$ such that   
$(Q, e_Q) \leq   (R, f)$ and  $ R \leq N_G(Q, e_Q) $. 
Let $$\alpha =  \siB^A_{N_P(Q), Q}(e_{N_P(Q)}). $$    By  
Lemma \ref{localcompatibility},  $(N_P(Q), \alpha)$  is a maximal 
$(A(Q), e_Q, N_G(Q, e_Q))$-Brauer pair.    Thus, 
by     Theorem \ref{relatingBrauerwithtrace} (b), 
$e_Q  \in \Tr{N_P(Q)} {N_G(Q,e_Q)}  (A(Q)^{N_P(Q)}) $.       
Since  $e_Q$ is central in   $A(Q)$, idempotent   and  an element of 
$A^{N_G(Q,e_Q)}$  
multiplying on both sides     by $e_Q$    gives that 
$$e_Q  \in  \Tr{N_P(Q)} {N_G(Q,e_Q)}  ( ( e_QA(Q)e_Q)^{N_P(Q)}) . $$   
Now, $C_G(Q, e_Q) $ is a  normal   subgroup of  $N_G(Q, e_Q)$  and   
since  $(A,b,G)$ is a saturated triple $e_Q$ 
is  a  primitive idempotent of $(A(Q))^{C_G(Q, e_Q)}  $ and  hence  also   of 
$(e_QA(Q)e_Q)^{C_G(Q, e_Q)}  $. Thus, by    Lemma \ref{fullynormalisedSylowgen} 
applied with  $B = e_QA(Q)e_Q$,  
$ H = N_G(Q, e_Q)$,  $C=C_G(Q,e_Q)$   and $R= N_P(Q)$,  we have that 
$N_P(Q)C_G(Q, e_Q)/C_G(Q, e_Q) $  
is a Sylow 
$p$-subgroup  of $N_G(Q, e_Q)/C_G(Q, e_Q) $.  
Since  $N_P(Q)C_G(Q, e_Q)/C_G(Q, e_Q) \cong  N_P(Q)/C_P(Q) \cong  \Aut_P(Q)$  
and  $N_G(Q, e_Q)/C_G(Q, e_Q)\cong  \Aut_{\Ff} (Q) $, it follows that  
$Q$ is fully $\Ff$-automised. 

It remains to show that $Q$ is $\Ff$-receptive. For this, 
we first observe that the hypothesis  on $Q$
implies that $(N_P(Q), e_{N_P(Q)}) $ is also maximal amongst  
$(A, b, G)$-Brauer pairs   $(R, f)$ such that   $(Q, e_Q) \leq   
(R, f)$ and $ R \leq  N_P(Q) C_G(Q, e_Q) $.  Hence, by Lemma  
\ref{localcompatibility},  now applied  with     $H=N_P(Q)(C_G(Q,e_Q) $,
$(N_P(Q), e_{N_P(Q)})$ contains   an $N_P(Q)C_G(Q,e_Q) $ conjugate of any 
$(A, b, G)$-Brauer pair   which contains $(Q,e_Q)$ and whose first 
component is contained in 
$N_P(Q)C_G(Q, e_Q)$.
Now  let $\varphi \, : \, R \to Q$ be an   isomorphism in $\Ff$, and let 
$ g \in G   $  induce $\varphi $, that is, 
$ \, ^g(R, e_R)  = (Q, e_Q)$ and $\varphi (x)=   gxg^{-1}$   for all $ x\in R$.
Then, it is an easy check that  $N_{\varphi} =  
N_P(R) \cap \,^{g^{-1}}N_P(Q)C_G(Q,e_Q) $.
Set $ N'=  \,^g N_{\varphi} =  \,^g N_P(R)  \cap N_P(Q)C_G(Q,e_Q) $,   
$e'_{N'}= \,^g{e_{ N_{\varphi}}} $  
and  consider the  $(A,b, G)$-Brauer  pair $ (N', e_N)$.  Since 
$(R, e_R) \leq (N_{\varphi}, e_{N_{\varphi}}) $, 
$(Q, e_Q) \leq  \,^g (N_{\varphi}, e_{N_{\varphi}}) = (N', e'_{N'})$.  Also, 
 $ N' \leq    N_P(Q) C_G(Q, e_Q)$. 
Thus, as pointed out above     $ \, ^h(N', e'_{N'}) \leq  
(N_P(Q), e_{N_P(Q)})$  for some 
$ h\in N_P(Q)C_G(Q, e_Q)$.  Multiplying by some element of $N_P(Q)$ 
if necessary,  we may assume that  $ h\in C_G(Q, e_Q)$.     
Since $\,^{hg}(N_{\varphi},  e_{N_{\varphi}})  \leq (P, e_P)$ and hence 
$\bar\varphi:=\conj_{hg} \, : \, N_{\varphi} \to  P$ is a  morphism in $\Ff$.  
and since   $ h \in C_G(Q,e_Q)$, $\bar\varphi $ extends $\varphi$.        
Thus $Q$ is $\Ff$-receptive.
\end{proof}
\

We now give the proof of Theorem \ref{saturationresult}. 
\

\begin{proof}  Keep the notation of the theorem, set  
$\Ff= \Ff_{ (P, e_P)}(A, b, G) $ and for each   
$Q\leq P$, let $e_Q$ be the unique block of $A(Q)$ such that    
$(Q, e_Q) \leq (P, e_P)$.    We have shown  
in Proposition \ref{permutationfusion}  that   
$\Ff$ is a fusion system on $P$.     Thus, by 
Lemma \ref{fulautrec}  nad by the saturation axioms of \cite{robertson-shpectorov} it suffices to show that 
 each subgroup of $P$ is 
$\Ff$-conjugate to a subgroup $Q$ of $P$ such that  $(N_P(Q), e_{N_P(Q)}) $ 
is maximal amongst  $(A, b, G)$-Brauer pairs   $(R, f)$ with    
$(Q, e_Q) \leq   (R, f)$ and $ R \leq N_G(Q, e_Q) $. So, let $ Q' \leq P$, 
and let  $(T,  \alpha) $ be  a maximal $(A(Q'), e_{Q'},  N_G(Q',e_{Q'}))$-
Brauer pair.  By  Lemma   \ref{Qtrivial} (c), $  Q'\leq T$.  
Let $f = {\siB^A_{R, Q}}^{-1}(\alpha)$. By Lemma \ref{localcompatibility},  $(T, f)$    
is an $(A, b, G)$-Brauer pair with $(Q', e_{Q'}) \leq (T, f)$.   
Since $(P,e_P)$ is a  maximal $(A,b, G)$-Brauer pair, we have
$$ \,^g(Q',e_{Q'}) \leq \,^g(T, f) \leq (P, e_P) $$ for some   
$ g\in G$.  Set $ Q =\, ^g Q'$. By the above,  
$\conj_g \,: \, Q' \to Q$  is  a morphism in $\Ff$, so 
$Q$ is  $\Ff$-conjugate to $Q'$. We will show that $(N_P(Q), e_{N_P(Q)})$  
has the required   maximality property. Note that by 
Lemma \ref{localcompatibility}, $(T, f)$ is  maximal amongst 
$(A,b,G)$-Brauer pairs which contain $(Q', e_{Q'})$ and whose first 
component is contained in $N_G(Q', e_{Q'})$. Thus, by transport of structure
  $\,^g(T, f)$ is maximal amongst 
$(A,b,G)$-Brauer pairs which contain $(Q, e_{Q})$ and whose first 
component is contained in $N_G(Q, e_{Q})$.    Since 
$\,^g(T, f) \leq (P, e_P) $,  $ \,^gT \leq N_P(Q) $ and 
$\,^g f =  e_{\,^gT} $.  Consequently, 
$ \,^g(T, f ) \leq (N_P(Q), e_{N_P(Q)})$. Since $(N_P(Q), e_{N_P(Q)})$  
contains $(Q, e_Q)$ and   $N_P(Q)  $  is contained in $N_G(Q,e_Q)$, 
the maximality of   $\,^g(T, f ) $ forces 
$ \,^g(T, f )= (N_P(Q), e_{N_P(Q)})$, and completes the proof of the theorem.
\end{proof}

\section{$p$-permutation modules  and saturation} \label{proofofmodsat}

Let  $G$ be a finite group,  $M$ an indecomposable $p$-permutation 
$kG$-module,   and $P$ a vertex of $M$ and set $A= \End_k(M)$.  
Then $A$ is a  $G$-algebra via  the  map 
$$  G \times  A \to   A, $$ sending the pair $(g, \phi)$ 
to the   element $\,^g\phi$   of $A$  defined by 
$$\,^g\phi(m)  =  g\phi(g^{-1}m), \   \   m\in M .$$ 
Since    $M$  is a $p$-permutation module, $M$ is a $p$-permutation 
$G$-algebra and since  $M$ is indecomposable, $1_A=id_M$ is  primitive in 
$ \End_k(M)^G$.  

\begin{prp}\label{connection}  With  the notation above, the
$(A, 1_A, G)$-Brauer pairs are the pairs  
$  (Q, 1_{A(Q)} ) $  such that $ M(Q) \ne 0 $ and  $(P, 1_{A(P)})$ 
is a maximal $(\End_k(M), 1_{\End_k(M)}, G)$-Brauer pair.  Further,

(i)   $\Ff_{(P, 1_{A(P)})} (A, 1_A, G) = \Ff_P(G) $.

(ii)  The triple $(A, 1_A, G)$ is of saturated type if and only if 
$M$ is Brauer indecomposable.

\end{prp}
  
\proof   Let $Q$ be a $p$-subgroup of $G$.  There is a natural action 
of $A(Q)$ on  $M(Q)$ which induces an isomorphism of $kN_G(Q)/Q$-algebras.  
between $A(Q)$ and $\End_k(M(Q))$
(see for instance \cite[Proposition 27.6]{thevenaz}). 
Since the identity element is the  only central idempotent of a matrix algebra, 
it follows that  the  $(A, 1_A, G)$-Brauer pairs are the pairs  
$  (Q, 1_{A(Q)} ) $  such that $ M(Q) \ne 0 $.  The maximality of 
$(P,1_{A(P)})$ is immediate from the fact that  $P$ is a vertex of $P$ 
and that $M(Q) \ne 0 $ if and only if  $Q$ is contained in a vertex of $M$ 
(see \cite[Corollary 27.6]{thevenaz}). Clearly,  
$\,^g1_{A(Q)} = 1_{A(\,^gQ)}$,  for any $ g\in G$ and  (i) is immediate 
from this. Under the  natural identification of  $A(Q)$ and $\End_k(M(Q))$  
$1_{A(Q)}= id_{M(Q)}$. Hence  $1_{A(Q)}$ is primitive in $ (A(Q))^{C_G(Q)} $ 
if and only if $M(Q)$ is an indecomposable $ kQC_G(Q)/Q$-module.   
The equivalence of (ii) is immediate  from this and the fact that $1_A$ 
is a central idempotent  of $A$ and hence of $A^G$.

\bigskip 

{\it Proof  of Theorem  \ref{saturationresultmod}.} 
In light of   Proposition  \ref{connection}, this is a special case of    
Theorem \ref{saturationresult}.

\bigskip

{\bf Remarks}  1. Let  $P$ be a $p$-subgroup  of $G$. Since  there 
exist indecomposable $p$-permutation $kG$-modules with vertex $P$, the 
analysis  before the statement  of    Theorem \ref{saturationresultmod}
shows that   given any $p$-subgroup $P$ of a finite group 
$G$, there exists a  $p$-permutation $G$-algebra $A$, 
and a primitive idempotent $b$ of $A^G$  such that  there is a maximal 
$(A,b,G)$-Brauer  pair, say $(P,e_P)$ with first component $P$ and  such that 
$\Ff_{(P,e_P)} (A, b, G ) = \Ff_P(G)$. On the other hand,  there exist  pairs 
$P,G $  where $G$ is a finite group and  $P$ is a $p$-subgroup of $G$ 
such that    $\Ff_P(G)$ is not a saturated system-for instance if 
$P$ is  a   non-Sylow $p$-subgroup of $G$ such that $N_{S}(P)$ strictly contains
$PC_S(P)$  for some Sylow $p$-subgroup  $S$ of $G$ containing $P$.   Thus, the  fusion system $\Ff_{(P, e_P)}(A,b, G)$ is not  always saturated.

2.  Suppose that   $b$ is a  (non-principal) block of $kG$ such that  a 
defect   group   $P$  of $kGb$   is a Sylow $p$-subgroups of $G$, but  $\Br_P^{kG}(b)$  is  a sum of more than one block of $kC_G(P)$. Let $M$  be  an indecomposable  $p$-permutation  module  $kG$-module in   the block $b$ and with vertex $P$.   Then,  since  $N_G(P) $ acts  transitively on  the set  ${\mathcal E} $  of blocks    $e$   of $kC_G(P) $   such that $\Br_P^{kG} (b) e = e $  and $M(P) \ne 0$,   $M(P)e \ne  0  $      for any $e \in {\mathcal E} $, and in particular, $M(P)$ is not indecomposable as    $kC_G(P)$-module.  However, since   $P$ is a Sylow $p$-subgroup  of $G$, $ \Ff_P(G)$  is  a  saturated fusion  system on $P$ (see \cite{brotolevioliver}).   Thus,   the converse of Theorem \ref{saturationresultmod} 
does not   hold in general.  Since Theorem \ref{saturationresultmod}  
is a special case    of Theorem \ref{saturationresult}, it follows also that 
the converse of  Theorem \ref{saturationresult} does not hold.
It    might be   that the methods of proof  of Theorem \ref{saturationresult}    can be refined   to  yield a
condition  on $(A,b,G)$   which  in   certain  situations (as in the one just discussed) is weaker than 
the condition of $(A,b,G)$   being a   saturated triple, and which in all cases  is necessary and sufficient    for 
the saturation of the corresponding fusion systems.

\iffalse

In particular, the converse of Theorem \ref{saturationresultmod} 
does not   hold in general.  Since Theorem \ref{saturationresultmod}  
is a special case    of Theorem \ref{saturationresult}, it follows also that 
 the converse of  Theorem \ref{saturationresult} does not hold.
\fi

\bigskip
We now prove Theorem \ref{Scottcyclic}. We  need some lemmas. 
The following is well known.

\begin{lmm}\label{initial}  Let $H$ be a finite group and $N$ 
a normal subgroup of $H$ 
such that  $H/N $ is a $p'$-group. Then, the   restriction of the 
projective cover of the trivial $kH$-module to $kN $ is indecomposable.
\end{lmm}

\proof  Under the hypothesis, $J(kG) =    J(kN) kH$.   Let $V$ be a projective
$kH$-module. Then,
$$  \Res{N}{H}  Rad(V) =  \Res{N}{H}  J(kH)V =  \Res{N}{H}  J(kN)kH V= 
\Res{N}{H}  J(kN)V   = Rad ( \Res{N}{H} V ) . $$
Consequently, 
$$  \Res{N}{H}  (V/ Rad(V) )=  \Res{N}{H}V/Rad ( \Res{N}{H} V ) . $$ 
The result is  immediate.

\bigskip

{\bf Remark.}  The  above indecomposability result holds for the 
projective cover of any simple  
$kH$-module whose restriction to  $N$ remains simple.

\bigskip

\begin{lmm}  \label{Sylowequiv} Let  $G$ be a finite group, 
$P$ a $p$-subgroup of $G$ and 
$M:= \Sc_P(G,k) $ the Scott module of  $kG$ relative to $P$.    

(i)  $M(P)$ is  
indecomposable as   $PC_G(P)/ P $-module if and  only if 
$N_G(P)/PC_G(P) $  is a $p'$-group.   

(ii) If $\Ff_P(G)$ is a 
saturated fusion system, then $M(P)$ is  
indecomposable as   $PC_G(P)/ P $-module.
\end{lmm} 

\proof   (i)   $M(P)$ is the  projective cover of the trivial 
$N_G(P)/P$-module and in particular is indecomposable as $kN_G(P)/P$-module.
The   forward implication  follows from Lemma \ref{fullynormalisedSylowgen}, 
applied with  $B=\End_k(M(P))$,  $H=N_G(P)$,  $R=P$ and  $C=C_G(P)$. 
The  backward  implication   is clear  from  Lemma  \ref{initial}. 
 
(ii) Suppose that  $\Ff_P(G)$ is a saturated fusion system. Then,  
$\Aut_P(P)$ is a Sylow $p$-subgroup of $\Aut_{\Ff}(P)$.  
On the other hand,  the  image of $\Aut_P(P)$ under the natural isomorphism  
from  $\Aut_{\Ff}(P)$     to $N_G(P)$     is $PC_G(P)/P$. Thus, 
$N_G(P)/PC_G(P) $ is a $p'$-group.      The result is immediate from (i).

\bigskip

\begin{lmm} \label{central} Let  $G$ be a finite group, 
$P$ a $p$-subgroup of $G$,
$M =\Sc_P(G,k) $ the Scott module of  $kG$ relative to $P$.  Suppose that 
$\Ff_P(G)$ is a saturated fusion system and let $ Q \leq Z(P)$.  
If $M(Q) $ is indecomposable as $kN_G(Q)/Q$-module, then   
$M(Q) $ is indecomposable as $kC_G(Q)/Q$-module.
\end{lmm}

\proof  Suppose    that 
$M(Q) $ is indecomposable as $N_G(Q)/Q$-module and set $L =N_G(Q)$ and 
$ C=C_G(Q)$.   Since $ Q\le Z(P)$  the extension axiom for  saturated 
fusion systems  implies   that 
$L=   C[N_G(P) \cap L] $. 
We consider $M(Q)$ as $kL$-module via inflation.
Since $M(Q)$ has vertex $P $ and $ P \leq C$, there exists 
an indecomposable  $p$-permutation   
$k C$-module $V$ with vertex $P$ such that   
$M(Q) $ is a direct summand of $\Ind{C}{L}V $.   
Let $W$ be an indecomposable  summand of $\Res{C}{L}\Ind{C}{L}V $.  
By the Mackey formula, $ W \cong \,^x V$ for some $ x\in L $.    
In particular, $ \, ^xP$ is  a vertex of $ \,^xV$.  
By the decomposition of $L$ given above,  $ x = uv  $ for some 
$u \in  C_G(Q)$, $ v\in N_G(P) $.  Thus, $ \, ^xP = \,^{u}P $ is 
$C$-conjugate to  $P$, and it follows that $P$ is a vertex of $W$. 
In particular, $W(P) \ne 0 $.
Let
$$ \Res{C}{L} M(Q) = W_1 \oplus \cdots \oplus W_s $$  be a decomposition of  
%$W$ 
%%NK
$M(P)$
as a  direct sum of indecomposable  $kC $-modules  and   suppose if 
possible that $ s > 1 $.  
 By the above argument, $W_i(P) \ne 0 $   for $i$, $1\leq i \leq s $, hence
$$ \Res{C\cap N_G(P)}{N_G(P)}M(P) \cong  (\Res{C}{L} M(Q))(P) =  
W_1(P) \oplus \cdots \oplus W_s(P) $$  is not indecomposable.  Since 
$C_G(P) \leq   C \cap N_G(P) $, it follows that $ \Res{C_G(P)}{N_G(P)}M(P)  $ 
is not indecomposable.   
But by the Sylow axiom for saturated fusion systems,   
$N_G(P)/PC_G(P) =\Aut_{\Ff}(P)$ is a  $p' $-group. This contradicts Lemma \ref
{Sylowequiv}.

\bigskip 

\

{\it Proof of  Theorem  \ref{Scottcyclic}.}   
%%NK
Let $M=\Sc_P(G,k)$.
Suppose   that  $\Ff:=\Ff_P(G)$  is saturated and let 
$Q \leq P$.  
We will show that $M(Q)$ is indecomposable as 
$kC_G(Q)$-module. We proceed by induction on the index of $Q$ in $P$. 
If $Q=P$, then  by Lemma \ref{Sylowequiv}, $M(Q)$ is indecomposable as  $kPC_G(P)/P$-module.   
Suppose now that $Q $ is proper in $P$ and that   $M(R)$ is indecomposable as 
$kRC_G(R)/R$-module   for any $p$-subgroup 
%%NK
$R$ 
of $P$ properly containing  $Q$.
Since $ P \leq N_G(Q)$,   $\Sc_P(N_G(Q), k)$ is a   direct summand  of  
$\Res{N_G(Q)}{G} M $
%%NK
(see \cite[Chapter 4, Theorem 8.6]{Nagao-Tsushima}). 
Write
$$\Res{N_G(Q)}{G} M = \Sc_P(N_G(Q),  k)  \oplus  X. $$
We claim that %$X(Q) \ne 0 $. 
%%NK
$X(Q) = 0 $. 
Indeed, suppose if possible that 
there exists a direct summand, say $N$ of $X$ such that $N(Q) \ne 0 $ and let 
$ R$ be a vertex of $N$.  Since $Q$ is normal    in $N_G(Q)$, we have 
that $Q \leq R$.   The  group $Q$ is not a vertex of the indecomposable   
$kG$-module $M$.  Hence by the Burry-Carlson-Puig theorem 
(see \cite[Chapter 4, Theorem 4.6 (ii)]{Nagao-Tsushima}), 
$\Res{N_G(Q)}{G} M  $ does 
%%NK
not  
have any indecomposable summand with vertex $Q$. 
Thus $ Q$ is a proper subgroup of $R$. On the other hand,  since 
$M$  is a summand of $\Ind{P}{G}k $, and $N$ is a summand of 
$\Res{N_G(Q)}{G} M $,  by the Mackey formula, $N$ is relatively   
$\,^x P \cap N_G(Q)$-projective  for some $x\in G$. Thus, 
$$ Q  < R < \,^x P   \text { \ and \ }    Q <  P. $$  
In particular, conjugation by 
$x$ is an $\Ff$-isomorphism from $\,^{x^{-1}}Q $ to  $Q$.     Now  $P$
is abelian and $\Ff$ is saturated.   So, by the extension axiom  
there exists  a $g\in N_G(P)$ such that  $gx^{-1} \in C_G(Q)$.   Setting 
$ h =gx^{-1}$, and conjugating all terms in the above by $h$, we get 
$$ Q= \,^hQ  < \,^hR < \,^{hx} P = \,^gP = P. $$
Since $h \in N_G(Q)$,  replacing  $R$ by $ \,^hR $, we may assume that 
$R \leq P$.       Since   $N$ is a summand of  $X$ and 
$ N(R) \ne 0 $, we have $X(R) \ne 0 $.
Since  $\Sc_P(N_G(Q),  k)$   has vertex   $P$ and $ R\leq P$, 
we also have that  $\Sc_P(N_G(Q),  k)(R)\ne 0$.
The equation
$$\Res{N_G(Q)}{G} M =    \Sc_P(N_G(Q),  k)  \oplus    X, $$
implies  that  $M(R)$ is not indecomposable as 
$k[N_G(Q)\cap N_G(R)]$-module. Since $ RC_G(R) \leq N_G(Q)\cap N_G(R)$,  
it follows that  $M(R)$ is not indecomposable as $kRC_G(R)$-module or 
equivalently as  $kRC_G(R)/R $-module, a contradiction.  
This proves the claim. Thus,
$$ M(Q) = \Sc_P(N_G(Q),  k)(Q)  \oplus  X(Q)  = \Sc_P(N_G(Q),  k) $$
as $kN_G(Q)$ and hence as $kN_G(Q)/Q$-module. In particular,
$ M(Q)$  is indecomposable as $kN_G(Q)/Q$-module. By Lemma \ref{central}, 
$ M(Q)$  is indecomposable as $kQC_G(Q)/Q$-module,  completing the proof.

\iffalse
Since $P$ is cyclic,   any subgroup 
of $P$ is central and weakly $\Ff$-closed. So the result is immediate from
Lemmas  \ref{central} and \ref{centralweakly}.
\fi

\

{\it Proof of  Corollary \ref{Scottcycliccor}.}  If  $G$ 
has cyclic Sylow $p$-subgroups,   then  it is   easy to see that $ \Ff_P(G)$ 
is  saturated for any  $p$-subgroup $P$  of $G$. 
The result is immediate from  the Theorem  \ref{Scottcyclic}.

\

{\it  Proof  of  Corollary \ref{bimodulescott}.}   With the  hypothesis  of   the statement, it is  immediate that 
$\Ff_{ \Delta (P)}  (G_1\times G_2)  \cong  \Ff_P(G_1) 
%%(P) 
$.    Thus,  since  $P$ is  Sylow in  
$G_1$,  $\Ff_P(G_1)$  and hence  $\Ff_{ \Delta (P)}  (G_1\times G_2)   $ is a   saturated fusion system  on $P$ (see \cite{brotolevioliver}).
The result follows   from 
Theorem  \ref{Scottcyclic}. 

\

Finally, we  prove Proposition  \ref{symevi}.  
For this we set up some   more notation and recall a few facts about 
Scott modules.
Let $ (K, \Oo , k)$-be a $p$-modular system 
(we  assume  here that $ k $  is an algebraic closure of the field of 
$p$-elements).    Let  $G=S_n $, and  let $P$ be a $p$-subgroup of 
$G$.  Let $M=\Sc_P(G,k)$   
be the $kG$-Scott module with  vertex   $P$ and  let 
$\tilde M   = \Sc_P(G,\Oo)$     be the   $\Oo G$-Scott module  
with  vertex $P$, so that $ M =  k \otimes_{\Oo} \tilde M $. 
Let  $\chi: \tilde M \to K$   be the  character of the $\Oo G$-module 
$\tilde M $.  Since $\tilde M$ is a
$p$-permutation  $ \Oo G$-module, 
for any $p$- element   $x$ of $G$,  
$\dim_k M( \langle x \rangle )=  \chi(x) $. In particular,  if $Q$ is a 
$p$-subgroup of $G$, then
$\dim_k M( Q )  \leq   \chi(x)  $ for any  element $x$ of $Q$, with 
equality if  $Q=\langle x \rangle $.

\bigskip

{\it Proof  of  Proposition \ref{symevi}.}    Suppose that 
$n\leq 6$ and   that $\Ff_P(G)$ is saturated.
We will show that $M(Q)$ is  indecomposable  as $kC_G(Q)/Q$-module 
for  every  subgroup $Q$ of $P$.  By Theorem \ref{Scottcyclic}, 
we may assume that $P$ is not  abelian. If $P$ is a Sylow $p$-subgroup of $G$, 
then $M=k $ \cite[Theorem 2.5]{broue1}  and the result is immediate.  
So, we may assume that $P$ is a non-abelian, non-Sylow $p$-subgroup of $G$. 
Consequently, $p=2 $,  $ n =6 $ and  
$P$ is isomorphic to  the dihedral group of order $8$.

By the Sylow axiom for  saturated 
fusion systems,  $PC_G(P)$ is a Sylow $2$-subgroup of $N_G(P)$.   So,
up to  $G$-conjugacy $P$ is one of 
\
$ \langle (1,2,3,4), (1,3)\rangle, $
\
$ \langle (1,2,3,4)(5,6), (1,3)\rangle $
\
$ \langle (1,2,3,4), (1,3)(5,6)\rangle $
\
or
\
$\langle (1,2,3,4)(5,6), (1,3)(5,6)\rangle. $

We will  show that  in each case above, $M$ is Brauer indecomposable. 
It can be checked  directly that    $\Ff_P(G)$ is saturated in each case above-
the second case corresponds to the nilpotent fusion system, the remaining three
correspond to  the  saturated fusion system  on $D_8$  
in which  the automorphism  of exactly one Klein-$4$ subgroup   
contains an element of order $3$.  However,
we do not prove saturation   as  by  Theorem \ref{saturationresultmod} 
this will follow after the  fact of Brauer indecomposability.

Before embarking on  our case by case analysis, we recall the  
$2$-decomposition matrix  of $S_6$ \cite[Page 414]{jameskerber}:

\[ \begin{array}{clcccc}
  &  &   1 & 4_1 & 4_2 &16 \\
  &   &  (1) & (5, 1) & (4,2) & (3,2,1) \\
1 & (6) & 1  & & & \\
5 & (5,1) & 1 & 1 & & \\
9 & (4,2) & 1 & 1 & 1 & \\
16 & (3,2,1) & & & & 1 \\
10 & (4,1^2) & 2 & 1 & 1 &\\
5 & (3^2) & 1 & & 1  & \\
10 & (3, 1^3) & 2 &1 & 1 & \\
5& (2^3) & 1 & & 1 & \\
9 & (2^2, 1^2) & 1 & 1 & 1 & \\
5& (2, 1^4)&  1 & 1 & & \\
1 & (1^6) &1 & & & 
\end{array}\]

{\bf\underline {Case:}}  $P=  \langle (1,2,3,4), (1,3)\rangle $.
Then $P$ is a Sylow $p$-subgroup of $S_5 $, naturally considered as a 
subgroup of  $S_6 $ as a one-point stabilizer,
whence  $\tilde M$ is a direct summand of  $\Ind{S_5}{S_6}  (\Oo) $  
(see for instance \cite[Theorem 2.5]{broue1}).
On the other hand by   \cite[Page 32]{craven1},  $M$  
has dimension $6$. So, $\tilde M = \Ind{S_5}{S_6}  (\Oo)$.
Now, if    $ u= (1,3) $, then 
$ \chi  (u) = 4 $ and 
if  $u  =(1,2 )(3,4) $ or  $u= (1,2,3,4) $ then  $\chi(u)= 2 $. Hence, 
it follows that  unless $ Q \le P $ is $G$-conjugate to 
$\langle (1,3) \rangle $,  the  dimension of $ M(Q) \leq 2  $ and if 
$ Q=\langle (1,3) \rangle $, then $M(Q) $ has dimension $4$.  
On the other hand, since $M(P)$   as $kN_G(P)/P $-module is the  
projective cover of the trivial module,
$M(P)$ has dimension at least $2$.  So, if  $ Q \le P $ is  not 
$G$-conjugate to 
$\langle (1,3) \rangle $,  then  for any  $R \le P$  containing  
$Q$  as a normal subgroup,
$M(Q) \cong M(R)$ as  $k(N_G(Q) \cap N_G(R))$-module, hence as  
$kC_G(R)$-modules.  Arguing inductively, it follows  that  
$M(Q)  \cong M(P) $ as $kC_G(P) $-modules.   By   Lemma \ref{Sylowequiv},   
$M(P) $ is indecomposable  as $kPC_G(P)/P$-module, hence as  $kC_G(P)$-module.
Since $C_G(P) \le QC_G(Q)$, it follows that $M(Q) $ is indecomposable as 
$kQC_G(Q)/Q$-module. 

Now suppose that $ Q =  \langle (1,3) \rangle $. Then 
$ M(Q) $ is a  $4$-dimensional   $p$-permutation $ kN_G(Q)$-module.
Let $V$ be an indecomposable $kN_G(Q)$-module summand of $M(Q)$ and let 
$ Q \le   R \le N_G(Q)$ be a vertex of  $M(Q)$. Then 
$$ M(R)=  M(Q) (R) \ne 0, $$ whence $ \,^gQ \le  \,^gR  \le P$ or 
$ \,^g R \le N_P(\,^gQ)$.   Since no transposition in $P$ is  central in $P$, 
$R$  has order at most $4$ (and  for some summand  $V$ exactly  $4$).   
Let $S$ be a Sylow $p$-subgroup of $N_G(Q)$ containing $R$.     Since 
$V $ is a direct summand   of $\Ind{R}{N_G(Q)} (k)$, the  Mackey formula and 
the Green indecomposability theorem imply that  any direct summand of  
 $\Res{S}{N_G(Q)} V$ is  isomorphic to  $\Ind{\,^xR \cap S}{S} k $ for  some 
$x \in N_G(Q)$. In particular,  the dimension of    $V$ is divisible 
by  the index of $R$ in $V$.
Since  the Sylow $p$-subgroups
of $C_G(Q)=N_G(Q)$ have order $16$ and $R$ has   order $8$,  
it follows that  $V$ has dimension divisible by $4$. 
Thus, $V=M(Q)$.  In particular, $M(Q)$ is indecomposable   as  $kN_G(Q)$, 
and  $N_G(Q)=C_G(Q)$.

{\bf \underline{ Case:}} $P=  \langle (1,2,3,4), (1,3)(5,6)\rangle $.
By \cite{craven1} $M $  has composition factors
$1_G$, $4_1 \oplus 4_2 $,  $ 1_G$.    An inspection 
of the decomposition matrix and   the character table of $S_6$ gives that 
$ \chi =  \chi_{(6)} + \chi_{(4,2)} $. Further, the values of $\chi $ 
on non-trivial  
$2$  elements of $G$  are as follows:
$$\chi((1,3))=4,  \  \chi((1,3)(2,4))=2, \    \chi((1,2)(3,4)(5,6))=4,$$  
$$  \chi((1,2,3,4))= 0, \  \chi((1,2,3,4)(5,6)) =  2 . $$

Since   $ C_G(P) /Z(P) $  contains  an element of order $2$,  
it follows  as in the previous case that
$M(Q)$ is indecomposable as $kC_G(Q)$-module  for any $p$-subgroup $Q$ of 
$G$ such that $M(Q)$ has dimension $2$. From the above character calculations, 
we  may assume that the only non-trivial elements  of $ Q$ 
are in the $G$-conjugacy class of  $(1, 3)$ and $(1,2)(3,4)(5,6) $  
and  in particular are non-central involutions in $P$.  If $ Q$ 
contains two such involutions, then $ Q  = P$, so we may assume  that either 
$Q =  \langle (1,3) \rangle  $ or  $Q =  \langle (1,2)(3,4) (5,6) \rangle  $. 
But now the result  follows   as above since both of these 
involutions are central in some Sylow $p$-subgroup  and  in both cases  
$M(Q)$ has dimension $4$.

{\bf\underline{ Case:}} $P=  \langle (1,2,3,4), (1,3)(5,6)\rangle $.
The   image of $P$ under the exceptional  non-inner automorphism of 
$S_6$  is  $S_6$-conjugate to   $\langle (1,2,3,4), (1,3)\rangle $. 
The   result follows
from Case 1 by transport of structure.

{\bf \underline{ Case:}} $P= \langle (1,2,3,4)(5,6), (1,3)(5,6)\rangle $. 
By \cite{craven1} $M$   is two dimensional with composition factors
$1_G$, $ 1_G$.   Since $M(P)$  has dimension at least $2$,  
$ M(Q)= M(P)=M$  for all $Q \leq P$. By Lemma \ref{Sylowequiv}, $M=M(P)$ is 
indecomposable  as $kPC_G(P)/P$-module.  Hence,   $M(Q)=M$
 is indecomposable as  $kQC_G(Q)/Q$-module  for all $Q\leq P$ as required.
This completes the proof of  Proposition \ref{symevi}.

\bigskip

{\bf   Concluding Remarks.} 
Given a saturated fusion system, $\Ff$ on a finite $p$-group $P$, 
Park has shown that   there exists a finite group $G$ with $P \leq G$ 
and such that $\Ff= \Ff_P(G)$ (cf.\cite{Park:realising}). We    pose the following question:

{\it Given a saturated fusion system  $\Ff$ on a finite $p$-group 
$P$, does there  exist  a  saturated triple $(A,b, G)$   such that 
 $\Ff=\Ff_{(P,e_P)}(A,b, G) $  for some   maximal $(A,b,G)$- Brauer pair $(P,e_P)$ ?
}

\

%%NK%%
{\bf Acknowledgments}
The second author
was partially  supported by the Ministry of Education, Culture, Sports and Technology (MEXT) Grant-in-Aid for Young Scientists (B)19740018, 2007-2009, and also Grant-in-Aid for 
Young Scientists (B) 22740025, 2010-2012.

\addtocontents{toc}{\contentsline {chapter}{Bibliographie}{62}}
%%%%%%%%%%%%%%%%%%%%%%%%%%%%%%%
\bibliographystyle{amsalpha}  %

\end{document}